\documentclass[12pt,a4paper]{amsart}
\usepackage[utf8]{inputenc}
\usepackage{amsmath}
\usepackage{dsfont}
\usepackage{amssymb}
\usepackage{hyperref}
\usepackage{cleveref}

\title{A Choice-free look at algebraic extensions of valuations}
\author{C\'edric A\"id}
\address{(C~.A~.) University of Antwerp, Department of Mathematics, Middelheim\-laan~1, 2020 Antwerpen, Belgium}
\email{Cedric.Aid@uantwerpen.be}
\date{2025-11-10}

\usepackage{amsthm}
\newtheorem{thm}{Theorem}[section]
\newtheorem{prop}[thm]{Proposition}
\newtheorem{cor}[thm]{Corollary}

\newtheorem{lem}[thm]{Lemma}
\theoremstyle{definition}
\newtheorem{de}[thm]{Definition}

\usepackage[backend=biber]{biblatex}
\bibliography{valuation_extensions}

\begin{document}
\begin{abstract}
In this paper, we study extensions of valuations over algebraic field extensions without the use of the Axiom of Choice.
We show a bijection between the extensions of a valuation and the maximal ideals of the relative integral closure of its valuation ring.
In the case of a finite extension, we show that these maximal ideals exist.
We conclude with an elementary proof of the fundamental inequality.
\smallskip

\noindent
{\sc{Classification (MSC 2020):}} 12J20, 13F30, 12F05
\medskip

\end{abstract}
\maketitle
\section{Introduction}
Let $K$ be a field and $v$ a valuation on $K$.
We will denote by $\mathcal{O}_v$ the valuation ring of $v$ and $\mathfrak{m}_v$ its maximal ideal.
Its value group will be denoted $\Gamma_v$ and its residue field $\kappa_v$.
We also assume all rings to be commutative.

Chevalley's extension theorem (\cite[Theorem 3.1.1]{engler2005valued}) tells us that for any field extension $L$ of $K$, there exists an extension of $v$ to $L$.
The proof of this theorem uses Zorn's lemma, which is equivalent to the Axiom of Choice.
We will show that if $L$ is a finite extension of $K$, we can construct an extension of $v$ to $L$ without using the Axiom of Choice.
We do this in two steps. The first step is giving a bijection between the maximal ideals of the relative integral closure of $\mathcal{O}_v$ in $L$ and the extensions of $\mathcal{O}_v$ to $L$.
For this first step, we will not require $L$ to be a finite extension of $K$, only algebraic.
The second step is then to give a description of these maximal ideals when $L$ is a finite extension of $K$.
We will prepare for this by studying some ideals of the relative integral closure of $\mathcal{O}_v$ in $L$.

As immediate consequences of our description of the extensions of $v$ to a finite extension $L$ of $K$,
we get a version of the weak approximation theorem for distinct extensions of $v$ and a weaker version of the fundamental inequality.
Using this version of the weak approximation theorem as a starting point, we will give an elementary proof of the full fundamental inequality (\Cref{fun}).
While older proofs exist, they use more advanced methods.
The proof by I.~S.~Cohen and O.~Zariski (\cite{Zariski1957}) uses composite valuations, the proof by P.\ Roquette (\cite{Roquette1958}) uses topology and composite valuations
and the proof in \cite{engler2005valued} uses Galois theory.
\section{Ideals}
Let $v$ be a valuation on a field $K$ and let $L$ be a finite extension of $K$.
Let $R$ be the relative integral closure of $\mathcal{O}_v$ in $L$.
In this section we will prove some properties of its Jacobson radical and maximal ideals.

We begin with a few definitions.
\begin{de}
Let $R$ be a commutative ring.
We call $R$ reduced if it has no nonzero nilpotent elements.
\end{de}
\begin{de}
Let $R$ be a commutative ring. We define the Jacobson radical of $R$ as
$$J(R)=\left\{r\in R\mid\forall s\in R:\, 1+rs\in R^{\times}\right\}.$$
\end{de}
Should one accept the Axiom of Choice, it is possible to prove that $J(R)$ is the intersection of all maximal ideals of $R$ (see \cite[page 751]{dummit2004}).
Without the Axiom of Choice, we can still prove that it is contained in the intersection of all maximal ideals.
\begin{lem}{\label{in}}
Let $R$ be a ring and $I$ a maximal ideal of $R$.
Then $J(R)\subseteq I$.
\end{lem}
\begin{proof}
Let $r\in R\setminus I$.
As $I$ is a maximal ideal of $R$, we can find $s\in R$ for which $1+rs\in I$.
It follows that $1+rs\notin R^{\times}$ and so $r\notin J(R)$.
\end{proof}
We prove some properties of the Jacobson radical.
\begin{prop}{\label{rad}}
Let $R$ be a commutative ring.
Then $J(R)$ is a radical ideal of $R$.
\end{prop}
\begin{proof}
First we will show that $J(R)$ is an ideal.
Let $r,r'\in J(R)$.
Let $s\in R$.
Then $1+rs$ is invertible in $R$, so we can find $t\in R$ such that $t(1+rs)=1$.
Now we have
$$t(1+(r+r')s)=1+r'st.$$
As both $s$ and $t$ are in $R$, so is $st$, which means $1+r'st$ is invertible in $R$.
This shows $1+(r+r')s$ is invertible. As this holds for arbitrary $s\in R$, it follows that $r+r'\in J(R)$.
Now let $r\in J(R)$ and $r'\in R$.
Let $s\in R$.
Then $r's\in R$, by which $1+rr's\in R^{\times}$.
Once again, this holds for arbitrary $s\in R$, so $rr'\in J(R)$.

Now we will show that $J(R)$ is radical.
It suffices to show that, for $r\in R$, if $r^2\in J(R)$, then $r\in J(R)$.
Let $r\in R$ such that $r^2\in J(R)$.
Take an arbitrary $s\in R$. Then
$$(1+rs)(1-rs)=1-r^2s^2=1+r^2\left(-s^2\right).$$
As $r^2\in J(R)$, this is invertible.
It follows that $(1+rs)$ is also invertible.
As this holds for arbitrary $s\in R$, we get that $r\in J(R)$.
\end{proof}

As $J(R)$ is a radical ideal of $R$, we know that $R/J(R)$ is reduced.
Our next step is to show that, if we take $R$ to be the relative integral closure of a valuation ring $\mathcal{O}_v$ in an extension of $K$,
then $R/J(R)$ is a $\kappa_v$-algebra.
\begin{lem}{\label{int}}
Let $L$ be an extension of $K$ and $v$ a valuation on $K$.
Let $R$ be the relative integral closure of $\mathcal{O}_v$ in $L$.
Then $$J(R)\cap \mathcal{O}_v=\mathfrak{m}_v.$$
\end{lem}
\begin{proof}
We have two inclusions to show.
First, we will show that \hbox{$J(R)\cap \mathcal{O}_v\subseteq \mathfrak{m}_v$.}
Take $r\in \mathcal{O}_v\setminus \mathfrak{m}_v$.
Then $r^{-1}\in \mathcal{O}_v\subseteq R$.
As $$1+r\left(-r^{-1}\right)=0,$$ which is not invertible, $r\notin J(R)$.

Now we will show that $\mathfrak{m}_v\subseteq J(R)\cap \mathcal{O}_v$.
Let $r\in \mathfrak{m}_v$.
This means that $v(r)>0$.
First we will show that $1+rs\in R^{\times}$ for all $s\in \mathcal{O}_v$.
For any $s\in \mathcal{O}_v$, we have that $v(s)\geq0$, and so $v(rs)>0$.
From this we obtain that $v(1+rs)=0$ and so $1+rs\in R^{\times}$.
Now we will show this for $s\in R$ not necessarily in $\mathcal{O}_v$.
Because $R$ is the relative integral closure of $\mathcal{O}_v$ in $L$,
we have that $s$ is integral over $\mathcal{O}_v$, whereby
we can find some $n\in\mathds{N}$ and $a_0,\ldots,a_{n-1}\in \mathcal{O}_v$ such that
$$s^n+\sum_{i=0}^{n-1}a_is^i=0.$$
Multiplying by $r^n$ gives us
$$(rs)^n+\sum_{i=0}^{n-1}r^{n-i}a_i(rs)^i=0.$$
Set $$f(X)=X^n+\sum_{i=0}^{n-1}r^{n-i}a_iX^i$$ and $g(X)=f(X-1)\in \mathcal{O}_v[X]$.
Then $1+rs$ is a root of $g$.
Write $g(X)=\sum_{i=0}^nc_iX^i$.
Note that $c_0\equiv f(X-1)\equiv f(-1)$ modulo $X$, and so
$$c_0=f(-1)=(-1)^n\left(1+r\sum_{i=0}^{n-1}a_ir^{(n-1)-i}(-1)^{n+i}\right).$$
Setting $t=\sum_{i=0}^{n-1}a_ir^{(n-1)-i}(-1)^{n+i}\in \mathcal{O}_v$, we have already shown that $1+rt$ is invertible.
Therefore, $c_0$ is also invertible.
This implies $$(1+rs)\left(-c_0^{-1}\sum_{i=1}^nc_i(1+rs)^i\right)=1$$
and so $1+rs$ is invertible.
This shows that $r\in J(R)$.
\end{proof}

We now give some properties of prime and maximal ideals of the relative integral closure of a valuation ring in an algebraic field extension.
These will help us describe the domain and image of the maps involved in the bijection.
\begin{prop}{\label{prime}}
Let $L$ be an algebraic extension of $K$ and $v$ a valuation on $K$.
Let $R$ be the relative integral closure of $\mathcal{O}_v$ in $L$.
Let $I$ be a prime ideal of $R$ satisfying $\mathcal{O}_v\cap I=\mathfrak{m}_v$.
Then $I$ is a maximal ideal of $R$.
\end{prop}
\begin{proof}
The condition $\mathcal{O}_v\cap I=\mathfrak{m}_v$ implies that $R/I$ is a $\kappa_v$-algebra.
We will show that it is a field.
Let $x\in R\setminus I$.
Because $R$ is integral over $\mathcal{O}_v$, we can find a monic polynomial $F\in \mathcal{O}_v[X]$ such that $F(x)=0$.
If we write $F(X)=\sum_{i=0}^na_iX^i$ with $n=\deg(F)$, then the polynomial $f(X)=\sum_{i=0}^n\overline{a_i}X^i$ is a nonzero polynomial over $\kappa_v$ having $\overline{x}$ as a root.
Let $j$ be the highest power of $X$ dividing $f$.
Suppose $j=n$. Then $a_i\in I$ for $i<n$.
This implies $x^n=-\sum_{i=0}^{n-1}a_ix^i\in I$. As $I$ is prime, we would have that $x\in I$, a contradiction.
Therefore $j<n$.
We can write $f(X)=X^jg(X)$ with $g(X)=\sum_{i=j}^n\overline{a_i}X^{i-j}$ non-constant.
As $I$ is prime and $\overline{x}\neq\overline{0}$, it follows that $g(\overline{x})=\overline{0}$.
By construction, $g\left(\overline{0}\right)\neq\overline{0}$, and so we can use $g$ to invert $\overline{x}$.

Now that we know $R/I$ is a field, it immediately follows that that $I$ is maximal.
\end{proof}
\begin{prop}{\label{ex}}
Let $L$ be an algebraic extension of $K$ and $v$ a valuation on $K$.
Let $R$ be the relative integral closure of $\mathcal{O}_v$ in $L$.
Let $I$ be a maximal ideal of $R$.
Then $I\cap \mathcal{O}_v=\mathfrak{m}_v$.
\end{prop}
\begin{proof}
Let $I$ be a maximal ideal of $R$.
According to \Cref{in}, $J(R)\subseteq I$.
By \Cref{int}, we then also have $\mathfrak{m}_v\subseteq I$.
We will now prove that $I\cap \mathcal{O}_v\subseteq\mathfrak{m}_v$.
Let $x\in \mathcal{O}_v\setminus \mathfrak{m}_v$.
Then $x$ is invertible in $\mathcal{O}_v$ and therefore also in $R$.
As $I$ is a proper ideal of $R$ it follows that $x\notin I$.
\end{proof}

\section{The Bijection}
Let $v$ be a valuation on a field $K$.
Given an algebraic extension $L$ of $K$, we seek to establish a bijection between the maximal ideals of the relative integral closure of $\mathcal{O}_v$ in $L$ and the extensions of $v$ to $L$.
As we are able to recover a valuation from its valuation ring (see \cite[Proposition 2.1.2]{engler2005valued}), we need only give a bijection between these maximal ideals and valuation rings extending $\mathcal{O}_v$.

First, we give a few lemmas we will need later.
\begin{lem}{\label{cl}}
Let $K$ be a field and $v$ a valuation on $K$ with valuation ring $\mathcal{O}_v$.
Then $\mathcal{O}_v$ is relatively integrally closed in $K$.
\end{lem}
\begin{proof}
See \protect{\cite[Theorem 3.1.3 (1)]{engler2005valued}}.
\end{proof}
We will also need this computational result.
\begin{lem}{\label{ax}}
Let $R$ be a ring, $S$ an algebraic extension of $R$ and \hbox{$x\in S$.}
If $\sum_{i=0}^na_ix^i=0$ with $a_0,\ldots,a_n\in R$, then $a_nx$ is integral over $R$.
\end{lem}
\begin{proof}
Multiplying the equality by $a_n^{n-1}$ gives us
$$\sum_{i=0}^na_n^{n-1}a_ix^i=0,$$
which we can rewrite as
$$(a_nx)^n+\sum_{i=0}^{n-1}a_n^{n-1-i}a_i(a_nx)^i=0.$$
This shows $a_nx$ is integral over $R$.
\end{proof}

The following result is adapted from \cite[Theorem III.1.2]{Fuchs2001}.
This will give the map from the maximal ideals to the valuation rings.
We will make use of localization at a prime ideal.
For a background on this, we refer the reader to \cite[Section 15.4]{dummit2004}.
\begin{prop}{\label{loc}}
Let $L$ be an algebraic extension of $K$ and $v$ a valuation on $K$.
Let $R$ be the relative integral closure of $\mathcal{O}_v$ in $L$.
If $I$ is a prime ideal of $R$ for which $\mathcal{O}_v\cap I=\mathfrak{m}_v$,
then $R_I$, the localization of $R$ at $I$, is a valuation ring of $L$ extending $\mathcal{O}_v$.
\end{prop}
\begin{proof}
Valuation rings of $L$ are characterized by the property that for every $x\in L^{\times}$, they contain at least one of $x$ or $x^{-1}$.
We will show that $R_I$ satisfies this property.

Let $x\in L^{\times}$.
We have that $\sum_{i=0}^na_ix^i=0$ for some $n\in\mathds{N}$ and $a_0,\ldots,a_n\in K$ with $a_0,a_n\neq0$.
Let $j$ be such that $$v(a_j)=\min\{v(a_i)\mid0\leq i\leq n\}.$$ Set $c_i=a_j^{-1}a_i$.
Then $\sum_{i=0}^nc_ix^i=0$ and $$\min\{v(c_i)\mid0\leq i\leq n\}=0,$$ by which $x$ is algebraic over $\mathcal{O}_v$.
Now set $$\alpha_j=\sum_{i=j}^{n}c_ix^{i-j}$$
for $0\leq j\leq n$. Then $$\alpha_{j-1}=x\alpha_j+c_{j-1}$$ and $$\alpha_{n}=c_{n}.$$
Set $\alpha_{n+1}=c_{n+1}=0$. Then $$\sum_{i=0}^{n+1}c_ix^i=0$$ and $$\alpha_{n}=x\alpha_{n+1}+c_{n}.$$
As $\alpha_{n+1}=0$, we have $x\alpha_{n+1}\in I$.
We will now proceed by reverse induction.

Let $1\leq j\leq n+1$ and suppose $x\alpha_j\in I$.
We can rewrite $$\sum_{i=0}^{n+1}c_ix^i=0$$
as
$$0=\sum_{i=0}^{j-1}c_ix^i+x^j\alpha_j=\sum_{i=0}^{j-2}c_ix^i+x^{j-1}(\alpha_jx+c_{j-1}).$$
Because $\alpha_jx\in I\subset R$ and $c_{j-1}\in R$, we get that $\alpha_{j-1}=\alpha_jx+c_{j-1}\in R$.
It follows from \Cref{ax} that $x\alpha_{j-1}$ is integral over $R$ and therefore also over $\mathcal{O}_v$.
As $R$ is the relative integral closure of $\mathcal{O}_v$ in $L$ and $x\alpha_{j-1}\in L$, we must also have that $x\alpha_{j-1}\in R$.
We now consider three cases:\par
\textit{Case $1$:} $v(c_{j-1})=0$\newline
Then $c_{j-1}\notin \mathfrak{m}_v=I\cap \mathcal{O}_v$.
Since $x\alpha_j\in I$ by hypothesis, $\alpha_{j-1}\notin I$ and so $\alpha_{j-1}$ is invertible in $R_I$.
We now have that $x=\frac{\alpha_{j-1}x}{\alpha_{j-1}}\in R_I$ and therefore we do not need to continue the induction.\par
\textit{Case $2$:} $x\alpha_{j-1}\notin I$\newline
Then $x\alpha_{j-1}$ is invertible in $R_I$.
We now have that $x^{-1}=\frac{\alpha_{j-1}}{\alpha_{j-1}x}\in R_I$ and therefore we do not need to continue the induction.\par
\textit{Case $3$:} $v(c_{j-1})>0$ and $x\alpha_{j-1}\in I$\newline
In this case, we see that the induction hypothesis applies to $j-1$ and we continue the induction.

As $\min\{v(c_i)\mid0\leq i\leq n\}=0$, case $1$ guarantees that the induction will stop.
When the induction stops we find that $x\in R_I$ or $x^{-1}\in R_I$.
This shows that $R_I$ is a valuation ring of $L$.

To show that $R_I$ is an extension of $\mathcal{O}_v$, we have to show that $R_I\cap K=\mathcal{O}_v$ and $IR_I\cap K=\mathfrak{m}_v$.
By construction, we have $\mathcal{O}_v\subseteq R\subseteq R_I$.
As we also have $\mathcal{O}_v\subseteq K$, this gives us the inclusion $\mathcal{O}_v\subseteq R_I\cap K$.
Next we show $\mathfrak{m}_v\subseteq IR_I\cap K$.
We have that
$$\mathfrak{m}_v=I\cap \mathcal{O}_v\subseteq I\subseteq IR_I,$$
which gives us the desired result.
Now we will show $R_I\cap K\subseteq \mathcal{O}_v$.
Let $x\in K\setminus \mathcal{O}_v$. Then $x^{-1}\in\mathfrak{m}_v\subseteq IR_I$.
The fact that $IR_I$ is a proper ideal of $R_I$ implies that $x\notin R_I$.
The last remaining inclusion to show is $IR_I\cap K\subseteq\mathfrak{m}_v$.
As we have already shown that $R_I\cap K=\mathcal{O}_v$, it suffices to show that $IR_I\cap \mathcal{O}_v\subseteq\mathfrak{m}_v$.
Let $x\in \mathcal{O}_v\setminus\mathfrak{m}_v$.
Then $x$ is invertible in $\mathcal{O}_v$.
This implies it is also invertible in $R_I$.
Therefore, $x\notin IR_I$.
\end{proof}
The next proposition will be used to define the map from the valuation rings to the maximal ideals.
\begin{prop}{\label{max}}
Let $L$ be an algebraic extension of $K$ and $v$ a valuation on $K$.
Let $R$ be the relative integral closure of $\mathcal{O}_v$ in $L$.
Let $w$ be an extension of $v$ to $L$.
Then $\mathfrak{m}_w\cap R$ is a maximal ideal of $R$.
\end{prop}
\begin{proof}
By \Cref{cl}, $\mathcal{O}_w$ is integrally closed in $L$.
As $\mathcal{O}_v\subseteq \mathcal{O}_w$, this gives us the chain of inclusions $\mathcal{O}_v\subseteq R\subseteq \mathcal{O}_w$.
Because $\mathfrak{m}_w$ is a maximal ideal of $\mathcal{O}_w$, it is also a prime ideal of $\mathcal{O}_w$.
It now follows from the inclusion $R\subseteq \mathcal{O}_w$ that $\mathfrak{m}_w\cap R$ is a prime ideal of $R$.
The inclusion $\mathcal{O}_v\subseteq R$ implies that $$(\mathfrak{m}_w\cap R)\cap \mathcal{O}_v=\mathfrak{m}_w\cap \mathcal{O}_v=\mathfrak{m}_v.$$
\Cref{prime} now tells us that $\mathfrak{m}_w\cap R$ is maximal.
\end{proof}
The next two results will be used to show the maps constructed above are each other's inverses.
\begin{lem}{\label{invloc}}
Let $R$ be a ring and $I$ a prime ideal of $R$.
We have that $IR_I\cap R=I$.
\end{lem}
\begin{proof}
The inclusion $I\subseteq IR_I\cap R$ is trivial.
We will prove the other inclusion.
Let $x\in IR_I\cap R$.
Then $x=\frac{a}{b}$ for some $a\in I$ and $b\in R\setminus I$.
Rewriting this gives us $bx=a\in I$.
As $I$ is a prime ideal of $R$ and $b\in R\setminus I$, we must have $x\in I$.
\end{proof}
\begin{prop}{\label{invmax}}
Let $L$ be an algebraic extension of $K$ and $v$ a valuation on $K$.
Let $R$ be the relative integral closure of $\mathcal{O}_v$ in $L$.
Let $w$ be an extension of $v$ to $L$.
Let $I=\mathfrak{m}_w\cap R$.
Then $R_I=\mathcal{O}_w$.
\end{prop}
\begin{proof}
We begin with the inclusion $R_I\subseteq \mathcal{O}_w$.
Let $x\in R_I$. Then $x=\frac{a}{b}$ for some $a\in R$ and $b\in R\setminus I$.
As $\mathcal{O}_w$ is integrally closed in $L$ and contains $\mathcal{O}_v$, it must also contain $R$.
This implies $w(a)\geq0$ and $w(b)\geq0$.
We also have that $b\notin I$ and therefore $w(b)=0$.
It follows that $w(x)=w(a)-w(b)\geq0$ and so $x\in \mathcal{O}_w$.

We will now prove $\mathcal{O}_w\subseteq R_I$ using methods similar to those used to prove \Cref{loc}.
Let $x\in \mathcal{O}_w$.
We have that $\sum_{i=0}^na_ix^i=0$ for some $n\in\mathds{N}$ and $a_0,\ldots,a_n\in K$ with $a_0,a_n\neq0$.
Let $j$ be such that $$v(a_j)=\min\{v(a_i)\mid0\leq i\leq n\}.$$ Set $c_i=a_j^{-1}a_i$.
Then $\sum_{i=0}^nc_ix^i=0$ and $$\min\{v(c_i)\mid0\leq i\leq n\}=0,$$ by which $x$ is algebraic over $\mathcal{O}_v$.
Now set $$\alpha_j=\sum_{i=j}^{n}c_ix^{i-j}$$
for $0\leq j\leq n$. Then $$\alpha_{j-1}=x\alpha_j+c_{j-1}$$ and $$\alpha_{n}=c_{n}.$$
Set $\alpha_{n+1}=c_{n+1}=0$. Then $$\sum_{i=0}^{n+1}c_ix^i=0$$ and $$\alpha_{n}=x\alpha_{n+1}+c_{n}.$$
As $\alpha_{n+1}=0$, we have $x\alpha_{n+1}\in I$.
We will now proceed by reverse induction.

Let $1\leq j\leq n+1$ and suppose $x\alpha_j\in I$.
It follows that $w(x\alpha_j)>0$.
We can rewrite $$\sum_{i=0}^{n+1}c_ix^i=0$$
as
$$0=\sum_{i=0}^{j-1}c_ix^i+x^j\alpha_j=\sum_{i=0}^{j-2}c_ix^i+x^{j-1}(\alpha_jx+c_{j-1}).$$
Because $\alpha_jx\in I\subset R$ and $c_{j-1}\in R$, we get that $\alpha_{j-1}=\alpha_jx+c_{j-1}\in R$.
It follows from \Cref{ax} that $x\alpha_{j-1}$ is integral over $R$ and therefore also over $\mathcal{O}_v$.
As $R$ is the relative integral closure of $\mathcal{O}_v$ in $L$ and $x\alpha_{j-1}\in L$, we must also have that $x\alpha_{j-1}\in R$.
We consider two cases:\par
\textit{Case $1$:} $v(c_{j-1})=0$\newline
In this case $w(x\alpha_j)>w(c_{j-1})$ and so $w(\alpha_{j-1})=0$.
This implies $\alpha_{j-1}\in R\setminus I$.
We now have that $x=\frac{x\alpha_{j-1}}{\alpha_{j-1}}\in R_I$ and can stop the induction.\par
\textit{Case $2$:} $v(c_{j-1})>0$\newline
In this case $w(\alpha_{j-1})\geq\min\{v(c_{j-1}),w(x\alpha_j)\}>0$ and so $\alpha_{j-1}\in I$.
It then follows that $x\alpha_{j-1}\in I$, allowing us to continue the induction.\par

As $\min\{v(c_i)\mid0\leq i\leq n+1\}=0$, the induction must stop at some point, giving us $x\in R_I$.
\end{proof}

We are now ready to state the main result of this section.
\begin{thm}{\label{bijection}}
Let $L$ be an algebraic extension of $K$ and $v$ a valuation on $K$.
Let $R$ be the relative integral closure of $\mathcal{O}_v$ in $L$.
There is a one to one correspondence between maximal ideals of $R$ and valuation rings of $L$ extending $\mathcal{O}_v$.
The maximal ideal $I$ corresponds to the valuation ring $R_I$ and the valuation ring $\mathcal{O}_w$ corresponds to the maximal ideal $\mathfrak{m}_w\cap R$.
\end{thm}
\begin{proof}
Let $\mathds{V}$ denote the set of valuation rings of $L$ extending $\mathcal{O}_v$
and let $M$ denote the set of maximal ideals of $R$.
By \Cref{max}, we can define $$\Psi:\mathds{V}\rightarrow M:\mathcal{O}_w\mapsto \mathfrak{m}_w\cap R.$$
By \Cref{ex}, we can use \Cref{loc} to define $$\Phi:M\rightarrow \mathds{V}:I\mapsto R_I.$$
It follows from \Cref{invloc} and \Cref{invmax} that $\Phi$ and $\Psi$ are inverses.
\end{proof}

\section{Finite extensions}
Let $v$ be a valuation on a field $K$ and let $L$ be a finite extension of $K$.
Let $R$ be the relative integral closure of $\mathcal{O}_v$ in $L$.
We know there is a bijection between the maximal ideals of $R$ and the extensions of $v$ to $L$.
In this section, we will describe the maximal ideals of $R$ by studying the $\kappa_v$-algebra $R/J(R)$.
We show that, even without Zorn's lemma, $R$ has at least one maximal ideal.
Using the bijection this shows that $v$ extends to $L$.

We begin by bounding the $\kappa_v$-dimension of $R/J(R)$.
\begin{prop}{\label{fin_dim}}
Let $L$ be a finite extension of $K$ and $v$ a valuation on $K$.
Let $R$ be the relative integral closure of $\mathcal{O}_v$ in $L$.
Then $$\dim_{\kappa_v}\left(R/J(R)\right)\leq\dim_K(L).$$
\end{prop}
\begin{proof}
Let $n=[L:K]$.
Take any $r_0,\ldots,r_n\in R$.
Then there exist some $a_0,\ldots,a_n\in K$ not all $0$ such that $$\sum_{i=0}^na_ir_i=0.$$
Let $j\in \{0,\ldots,n\}$ such that $v(a_j)=\min\{v(a_i)|0\leq i\leq n\}$.
Then $a_j\neq0$. Set $c_i=a_j^{-1}a_i$. We have $$\sum_{i=0}^nc_ir_i=0$$ and $\min\{v(c_i)|0\leq i\leq n\}=0$.
This means that $c_ir_i\in R$ for all $i$.
Denoting by $\overline{r_i}$ the residue class of $r_i$ in $R/J(R)$, we get $$\sum_{i=0}^n\overline{c_i}\ \overline{r_i}=\overline{0}$$
and not all $\overline{c_i}$ are $\overline{0}$ because not all $c_i$ are in $\mathfrak{m}_v=J(R)\cap\mathfrak{m}_v$.
\end{proof}
We will need the following result to decompose $R/J(R)$.
Doing so will allow us to obtain maximal ideals of $R$.
\begin{prop}{\label{product}}
Let $F$ be a field and $A$ a finite-dimensional commutative \hbox{$F$-algebra.}
If $A$ is reduced, then there exist $r\geq1$ and $F_1,\ldots,F_r$ finite field extensions of $F$ such that
$$A\cong F_1\times\ldots\times F_r.$$
\end{prop}
\begin{proof}
We will prove this by induction on $\dim_F(A)$.
If $\dim_F(A)=1$, then $A=F$.

Now let $n\in\mathds{N}$, $n\geq1$ and suppose we have shown the lemma when $\dim_F(A)<n$.
We will show the case where $\dim_F(A)=n$.

If every nonzero element of $A$ is invertible, then $A$ is a field.
Since $F\subseteq A$ and $\dim_F(A)=n<\infty$, we have that $A$ is a finite field extension of $F$.

Suppose now that there exists some $a\in A\setminus\{0\}$ that is not invertible.
As $\dim_F(A)=n$, we can find $b_1,\ldots,b_{n+1}\in F$ not all $0$ such that $$\sum_{i=1}^{n+1}b_ia^i=0.$$
Let $j$ be minimal such that $b_j\neq0$. Set $c_i=b_j^{-1}b_i$.
Then $$a^j\sum_{i=j}^{n+1}c_ia^{i-j}=0.$$
Set $f(X)=\sum_{i=j}^{n+1}c_iX^{i-j}$.
Then $f(0)=c_j=1$ and $a^jf(a)=0$. Suppose $f(a)=0$. We would have $$a\left(-\sum_{i=j+1}^{n+1}c_ia^{i-(j+1)}\right)=1,$$ which is a contradiction as $a$ is not invertible.
Therefore $f(a)\neq0$. As $a^jf(a)=0$, we also have that $(af(a))^j=a^j(f(a))^j=0$.
Because $A$ is reduced, $af(a)=0$.
We can now compute
$$(f(a))^2=f(a)c_j+f(a)\left(a\sum_{i=j+1}^{n+1}c_ia^{i-(j+1)}\right)=f(a).$$
This means $f(a)$ is idempotent.
Note that $f(a)$ is not invertible as $af(a)=0$.
Therefore $Af(a)$ is a proper ideal of $A$.
Furthermore, as $f(a)$ is idempotent, $$\phi:A\to Af(a):x\mapsto xf(a)$$ is a ring homomorphism.
As all nonzero elements of $F$ are invertible, the image of $F$ under $\phi$ is isomorphic to $F$. Therefore $\phi$ is an $F$-algebra homomorphism.
Let $\pi:A\to A/Af(a)$ be the projection. This is also an $F$-algebra homomorphism.
Now define $$\psi:A\to Af(a)\times(A/Af(a)):x\mapsto(\phi(x),\pi(x)).$$
Both components of $\psi$ are $F$-algebra homomorphisms, and therefore $\psi$ is an $F$-algebra homomorphism.
We show that $\psi$ is injective.
Let $b\in \ker(\psi)$. Then $b\in Af(a)$. This means $b=cf(a)$ for some $c\in A$. This implies $$0=bf(a)=c(f(a))^2=cf(a)=b.$$
This means $\psi$ is injective, and as $$\dim_F(Af(a))+\dim_F(A/Af(a))=\dim_F(A)<\infty,$$ $\psi$ is bijective.
Then $\psi$ is an $F$-algebra isomorphism.
Recall that $Af(a)$ is a proper ideal of $A$. This means $$dim_F(Af(a))<\dim_F(A)=n$$ and $$\dim_F(A/Af(a))<\dim_F(A)=n.$$

The last step before we can apply the induction hypothesis is to show that $Af(a)$ and $A/Af(a)$ are reduced.
This follows immediately from the fact that $A$ is reduced and $$A\cong Af(a)\times A/Af(a).$$
The induction hypothesis then completes the proof.
\end{proof}
The next result will allow us to make use of the decomposition of $R/J(R)$ to describe the maximal ideals of $R$.
\begin{prop}{\label{semi}}
Let $R$ be a ring.
If $R/J(R)\cong F_1\times\ldots\times F_n$ for some fields $F_1,\ldots,F_n$, then $R$ is semilocal.
Its maximal ideals are given by $I_i=\ker(\pi_i)$ with $\pi_i$ the projection from $R$ to $F_i$ for $i=1,\ldots,n$.
\end{prop}
\begin{proof}
Set $I_i=\ker(\pi_i)$ for $i=1,\ldots,n$.
Then $R/I_i\cong F_i$, which is a field, and therefore $I_i$ is a maximal ideal of $R$.

Now let $I$ be an arbitrary maximal ideal of $R$.
We will show that $I\subseteq I_j$ for some $j$.
Suppose that this is not the case, then, for each $i$, we can find an $x_i\in I$ such that $x_i\notin I_i$.
Then $\pi_i(x_i)\neq0$.
As $R/J(R)\cong F_1\times\ldots\times F_n$, for any $i$, we can find $y_i\in R$ such that $\pi_i(y_i)=(\pi_i(x_i))^{-1}$ and $\pi_{i'}(y_i)=0$ for $i'\neq i$.
Then $$\sum_{i=1}^nx_iy_i\in I$$ because $I$ is an ideal of $R$. We also have that $$\sum_{i=0}^nx_iy_i\equiv1$$ modulo $J(R)$.
This means we can find an $r\in J(R)$ such that $$\sum_{i=0}^nx_iy_i=1+r.$$
Because $I$ is a maximal ideal of $R$, we have that $J(R)\subseteq I$ and therefore this implies $1\in I$.
This is a contradiction as maximal ideals are proper.
It now follows that $I\subseteq I_j$ for some $j$.
The maximality of $I$ now implies that $I=I_j$.
\end{proof}
Putting these results together, we obtain a choice-free description of all extensions of $v$ to $L$.
Sadly, this description does not give an algorithm to explicitly compute the extensions.
It does, however, guarantee the existence of at least one extension of $v$ to $L$.
\begin{thm}{\label{ext}}
Let $L$ be a finite extension of $K$ and $v$ a valuation on $K$.
Let $R$ be the relative integral closure of $\mathcal{O}_v$ in $L$.
Then $v$ has finitely many extensions $w_1,\ldots,w_n$ to $L$.
Furthermore, $$R/J(R)\cong \kappa_{w_1}\times\ldots\times \kappa_{w_n},$$
where the isomorphism is given by $$r+J(R)\mapsto (r+\mathfrak{m_{w_1}},\ldots,r+\mathfrak{m_{w_n}}).$$
\end{thm}
\begin{proof}
By \Cref{fin_dim}, we know that $R/J(R)$ is a finite dimensional $\kappa_v$-algebra.
Furthermore, as $J(R)$ is radical, $R/J(R)$ is reduced.
It follows from \Cref{product} that $$R/J(R)\cong F_1\times\ldots\times F_n$$
for some field extensions $F_1,\ldots,F_n$ of $\kappa_v$.
We now apply \Cref{semi} and get all maximal ideals $I_1,\ldots,I_n$ of $R$
as the kernels of the projections $\pi_1,\ldots,\pi_n$ from $R$ to $F_1,\ldots,F_n$.
By \Cref{bijection}, we now obtain all extensions of $\mathcal{O}_v$ to $L$ as the localizations $R_{I_i}$.
Note that $$R_{I_i}/I_iR_{I_i}\cong R/I_{i}.$$
Let $w_i$ be the extension of $v$ corresponding to $R_{I_i}$.
As $\mathfrak{m}_{w_i}\cap R=\ker(\pi_i)$, the first isomorphism theorem (see \cite[page 243]{dummit2004}) tells us that $\pi_i$ descends to an isomorphism between $\kappa_{w_i}$ and $F_i$.
Combining these, we get the requested isomorphism.
\end{proof}

We will now state two immediate consequences of this theorem.
\begin{cor}{\label{weak}}
Let $v$ be a valuation on a field $K$ and let $L$ be a finite extension of $K$.
Let $w_1,\ldots,w_n$ be distinct extensions of $v$ to $L$.
Let $a_i\in\kappa_{w_i}$ for $1\leq i\leq n$.
There exists an $x\in L$ integral over $\mathcal{O}_v$ such that $x+\mathfrak{m}_w=a_i$ for $1\leq i\leq n$.
\end{cor}
\begin{proof}
Let $R$ be the relative integral closure of $\mathcal{O}_v$ in $L$.
By \Cref{bijection}, the map
$$r+J(R)\mapsto (r+\mathfrak{m_{w_1}},\ldots,r+\mathfrak{m_{w_n}})$$
can be inverted.
Taking any $x$ such that $x+J(R)$ is the inverse image of $(a_1,\ldots,a_n)$ gives us the element we seek.
\end{proof}
This first corollary is a version of the weak approximation theorem (see \cite[Theorem 3.2.7]{engler2005valued}) that applies only to distinct extensions of one valuation.
\begin{cor}
Let $v$ be a valuation on a field $K$ and let $L$ be a finite extension of $K$.
Let $w_1,\ldots,w_n$ be distinct extensions of $v$ to $L$.
Then
$$\sum_{i=1}^n[\kappa_{w_i}:\kappa_v]\leq [L:K].$$
\end{cor}
\begin{proof}
Let $R$ be the relative integral closure of $\mathcal{O}_v$ in $L$.
By \Cref{ext} and \Cref{fin_dim},
$$\sum_{i=1}^n[\kappa_{w_i}:\kappa_v]=\dim_{\kappa_v}(R/J(R))\leq[L:K].$$
\end{proof}
This second corollary is a weaker version of the fundamental inequality (see \cite[Theorem 3.3.4]{engler2005valued}).
In the next section, we shall give a proof of the full inequality using elementary methods.

\section{The fundamental inequality}
In this section, we will give a proof of the fundamental inequality using elementary methods.
We start by defining the quantities involved in the inequality.
\begin{de}
Let $v$ be a valuation on a field $K$ and let $L$ be an extension of $K$.
Let $w$ be an extension of $v$ to $L$.
We define the ramification index
$$e(w|v)=[\Gamma_w:\Gamma_v]$$
and the residue degree
$$f(w|v)=[\kappa_w:\kappa_v].$$
\end{de}
These quantities are only meaningful when they are finite.
If we are working over a finite extension, we can give an upper bound.
This corresponds to the fundamental inequality in the case the valuation on the base field extends uniquely.
\begin{lem}{\label{ef}}
Let $v$ be a valuation on a field $K$ and let $L$ be a finite extension of $K$.
Let $w$ be an extension of $v$ to $L$.
Let $a_1,\ldots,a_n\in \mathcal{O}_w$ such that $w(a_1)=\ldots=w(a_n)=0$ and $\overline{a_1},\ldots,\overline{a_n}$ are linearly independent over $\kappa_v$.
Let $b_1,\ldots,b_m\in L$ such that $w(b_1),\ldots,w(b_m)$ are in distinct equivalence classes modulo $\Gamma_v$.
Then, for any $c_{1,1},\ldots,c_{n,m}\in K$, we have that
$$w\left(\sum_{i=1}^n\sum_{j=1}^mc_{i,j}a_ib_j\right)=\min\left\{w(c_{i,j}a_ib_j)\mid 1\leq i\leq n,\, 1\leq j\leq m\right\}.$$
In particular, $\{a_ib_j\mid 1\leq i\leq n,\, 1\leq j\leq m\}$ are linearly independent.
\end{lem}
\begin{proof}
We adapt the proof from \cite[Lemma 3.2.2]{engler2005valued} for the sake of completeness.

We prove this in two steps.
First, fix an index $j$.
We will show
$$w\left(\sum_{i=1}^nc_{i,j}a_i\right)=\min\{v(c_{i,j})\mid1\leq i\leq n\}.$$
If $c_{1,j}=\ldots=c_{n,j}=0$, then both sides evaluate to $\infty$.
We now move on to the case they are not all $0$.
Then there exists some index $k$ such that
$$w\left(c_{k,j}\right)=\min\{v(c_{i,j})\mid1\leq i\leq n\}$$
and $c_{k,j}\neq0$.
For $1\leq i\leq n$, we set $d_i=c_{k,j}^{-1}c_{i,j}$.
Then $$\min\{w(d_i)\mid1\leq i\leq n\}=0.$$
As $d_i\in K$, it follows that $\overline{d_i}\in\kappa_v$, and so
$$\overline{\sum_{i=1}^nd_ia_i}=\sum_{i=1}^n\overline{d_i}\,\overline{a_i}\neq\overline{0}$$
because the $\overline{a_i}$ are $\kappa_v$-linearly independent.
Therefore $w\left(\sum_{i=1}^nd_ia_i\right)=0$.
Adding $v(c_{k,j})$ to both sides now gives the desired result, and we can continue to step two.

The second step is proving the equality.
If $c_{i,j}=0$ for all $1\leq i \leq n$ and $1\leq j\leq m$, then both sides evaluate to $\infty$.
We now consider the case when not all $c_{i,j}$ are $0$.
Then, for some $j$, we have that $w(\sum_{i=1}^nc_{i,j}a_ib_j)\in\Gamma_v+w(b_j)$.
If $j_1$ and $j_2$ are distinct indices for which this holds, then
$$w\left(\sum_{i=1}^nc_{i,{j_1}}a_ib_{j_1}\right)\in\Gamma_v+w(b_{j_1})$$
and
$$w\left(\sum_{i=1}^nc_{i,{j_2}}a_ib_{j_2}\right)\in\Gamma_v+w(b_{j_2}).$$
As $\Gamma_v+w(b_{j_1})\cap\Gamma_v+w(b_{j_2})=\varnothing$, their values must be distinct.
This implies
\begin{eqnarray*}
w\left(\sum_{i=1}^n\sum_{j=1}^mc_{i,j}a_ib_j\right)&=&\min\left\{w\left(\sum_{i=1}^nc_{i,j}a_ib_j\right)\mid 1\leq j\leq m\right\}\\
&=&\min\{w(c_{i,j}a_ib_j)\mid 1\leq i\leq n,\, 1\leq j\leq m\}.
\end{eqnarray*}

When not all $c_{i,j}$ are $0$, this minimum will not be $\infty$ and so $$\sum_{i=1}^n\sum_{j=1}^mc_{i,j}a_ib_j\neq0,$$
proving the $K$-linear independence of the $a_ib_j$.
\end{proof}
\begin{cor}
Let $v$ be a valuation on a field $K$ and let $L$ be a finite extension of $K$.
Let $w$ be an extension of $v$ to $L$. We have that
$$e(w|v)f(w|v)\leq [L:K].$$
\end{cor}
\begin{proof}
If both $e(w|f)$ and $f(w|v)$ are finite, then, by \Cref{ef}, we can find $e(w|v)f(w|v)$ $K$-linearly independent elements of $L$.
If at least one of $e(w|f)$ or $f(w|v)$ is infinite, then we instead find  arbitrarily many $K$-linearly independent elements of $L$.
It now follows that $e(w|v)f(w|v)\leq [L:K]$.
\end{proof}
It is this approach that we will use to prove the full fundamental inequality.
We will, however, need to be slightly more careful in our choice of elements, as well as argue that the elements we want to choose actually exist.
In order to do this, we need a few auxiliary results.
\begin{lem}
Let $v$ be a valuation on a field $K$ and let $L$ be an algebraic extension of $K$.
Let $w$ be an extension of $v$ to $L$.
Then $\Gamma_w/\Gamma_v$ is torsion.
\end{lem}
\begin{proof}
Let $x\in L$.
As $L$ is an algebraic extension of $K$, there exist $n\in\mathds{N}$ and $a_0,\ldots,a_n\in K$ not all $0$ such that
$$\sum_{i=0}^na_ix^i=0.$$
It follows that $$w\left(\sum_{i=0}^na_ix^i\right)=\infty>\min\{v(a_i)+iw(x)\mid0\leq i\leq n\}.$$
This implies that $v(a_i)+iw(x)=v(a_j)+jw(x)$ for some $i\neq j$.
Therefore, $(i-j)w(x)=v(a_j)-v(a_i)\in\Gamma_v$.
\end{proof}
Given a valuation $v$ on a field $K$ and an algebraic extension $L$ of $K$,
this allows us to view the value group of any extension of $v$ to $L$ as a subgroup of the divisible closure of $\Gamma_v$.
In this way, we can compare the values of an element of $L$ under different extensions of $v$.
\begin{lem}{\label{approx}}
Let $v$ be a valuation on a field $K$ and let $L$ be a finite extension of $K$.
Let $w_1,\ldots,w_n$ be distinct extensions of $v$ to $L$.
For any $\gamma\in\Gamma_{w_1}$, there exists an $x\in L$ such that $w_1(x)=\gamma$ and $w_i(x)>\gamma$ for $i>1$.
\end{lem}
\begin{proof}
As $\gamma\in\Gamma_{w_1}$, there exists some $x\in L$ such that $w_1(x)=\gamma$.

By \Cref{weak}, we can find $a\in L$ for which
\begin{itemize}
\item $w_1(a)=0$
\item $w_i(a)=0$ for $i>1$ for which $w_i(x)<\gamma$
\item $w_i(a)>0$ for $i>1$ for which $w_i(x)\geq\gamma$
\end{itemize}
We can also find $b\in L$ such that $w_1(b)>0$ and $w_i(b)=0$ for $i>1$.
Let $e$ be the order of $\gamma$ modulo $\Gamma_v$.
Then there exists a $c\in K$ such that $v(c)=2e\gamma$.
We have that $e\geq1$ and so $1-2e<0$.
This implies that multiplying an inequality with $1-2e$ reverses it.
As such, we see that the values of $bc(ax)^{1-2e}$ satisfy the following:
\begin{itemize}
\item $w_1\left(bc(ax)^{1-2e}\right)>\gamma$
\item $w_i\left(bc(ax)^{1-2e}\right)>\gamma$ for $i>1$ for which $w_i(x)<\gamma$
\item $w_i\left(bc(ax)^{1-2e}\right)<\gamma$ for $i>1$ for which $w_i(x)\geq\gamma$
\end{itemize}
Set $y=x+bc(ax)^{1-2e}$.
We have that $w_1(y)=\gamma$ and $w_i(y)<\gamma$ for $i>1$.
Now consider $cy^{1-2e}$.
We compute $$w_1\left(cy^{1-2e}\right)=2e\gamma+(1-2e)\gamma=\gamma$$
and for $i>1$
$$w_i\left(cy^{1-2e}\right)=2e\gamma+(1-2e)w_i(y)>\gamma.$$
The element $cy^{1-2e}$ has the required properties.
\end{proof}
We can now begin the proof of the fundamental inequality.
\begin{thm}{\label{fe}}
Let $v$ be a valuation on a field $K$ and let $L$ be a finite extension of $K$.
Let $w_1,\ldots,w_n$ be distinct extensions of $v$ to $L$.
Set $e_i=e(w_i|v)$ and $f_i=f(w_i|v)$.
For $1\leq i\leq n$, let $a_{i,1},\ldots,a_{i,f_i}\in L$ satisfying $w_i(a_{i,j})=0$ and $w_{i'}(a_{i,j})>0$ for $i'\neq i$ such that $a_{i,1}+\mathfrak{m}_{w_i},\ldots,a_{i,f_i}+\mathfrak{m}_{w_i}$ are linearly independent over $\kappa_v$.
For $1\leq i\leq n$, let $b_{i,1}\ldots,b_{i,e_i}\in L$ such that the $w_i(b_{i,j})$ represent distinct equivalence classes of $\Gamma_{w_i}/\Gamma_v$ and $w_{i'}(b_{i,j})\geq w_i(b_{i,j})$ for $i'\neq i$.

For any $$\{c_{i,j,k}\in K\mid 1\leq i\leq n,\, 1\leq j\leq f_i,\, 1\leq k\leq e_i\}$$
we have that
\begin{eqnarray*}
&\min\left\{w_l\left(\sum_{i=1}^n\sum_{j=1}^{f_i}\sum_{k=1}^{e_i}c_{i,j,k}a_{i,j}b_{i,k}\right)\mid1\leq l\leq n\right\}\\
&=\min\left\{v(c_{i,j,k})+w_i(b_{i,k})\mid1\leq i\leq n,\, 1\leq j\leq f_i,\, 1\leq k\leq e_i\right\}.
\end{eqnarray*}
In particular, the elements
$$\{a_{i,j}b_{i,k}\mid1\leq i\leq n,\, 1\leq j\leq f_i,\, 1\leq k\leq e_i\}$$
are linearly independent over $K$.
\end{thm}
\begin{proof}
For any fixed index $i$, it follows from \Cref{ef} that
\begin{eqnarray*}
w_i\left(\sum_{j=1}^{f_i}\sum_{k=1}^{e_i}c_{i,j,k}a_{i,j}b_{i,k}\right)&=&\min\left\{w_i(c_{i,j,k}a_{i,j}b_{i,k})\mid1\leq j\leq f_i,\, 1\leq k\leq e_i\right\}\\
&=&\min\left\{v(c_{i,j,k})+w_i(b_{i,k})\mid1\leq j\leq f_i,\, 1\leq k\leq e_i\right\}.
\end{eqnarray*}
We will set $\delta_i$ equal to this value.
Note that, for $i'\neq i$,
\begin{eqnarray*}
w_{i'}\left(\sum_{j=1}^{f_i}\sum_{k=1}^{e_i}c_{i,j,k}a_{i,j}b_{i,k}\right)&\geq&\min\left\{w_{i'}(c_{i,j,k}a_{i,j}b_{i,k})\mid1\leq j\leq f_i,\, 1\leq k\leq e_i\right\}\\
&>&\min\left\{v(c_{i,j,k})+w_i(b_{i,k})\mid1\leq j\leq f_i,\, 1\leq k\leq e_i\right\}\\
&=&\delta_i.
\end{eqnarray*}
As such, for any $i$, we have that
$$w_i\left(\sum_{i=1}^n\sum_{j=1}^{f_i}\sum_{k=1}^{e_i}c_{i,j,k}a_{i,j}b_{i,k}\right)\geq\min\{\delta_j\mid1\leq j\leq n\}.$$

Now take an index $l$ such that $\delta_l$ is minimal.
For this $l$ and for $i\neq l$, we have that
$$w_l\left(\sum_{j=1}^{f_i}\sum_{k=1}^{e_i}c_{i,j,k}a_{i,j}b_{i,k}\right)>\delta_i\geq\delta_l.$$
This implies that
$$w_l\left(\sum_{i=1}^n\sum_{j=1}^{f_i}\sum_{k=1}^{e_i}c_{i,j,k}a_{i,j}b_{i,k}\right)=\delta_l.$$
This shows that $$\delta_l=\min\left\{w_m\left(\sum_{i=1}^n\sum_{j=1}^{f_i}\sum_{k=1}^{e_i}c_{i,j,k}a_{i,j}b_{i,k}\right)\mid1\leq m\leq n\right\}.$$
We now compute the value we assigned to $\delta_l$.
\begin{eqnarray*}
\delta_l&=&\min\{\delta_i\mid1\leq i\leq n\}\\
&=&\min\left\{\min\{v(c_{i,j,k})+w_i(b_{i,k})\mid1\leq j\leq f_i,\, 1\leq k\leq e_i\}\mid1\leq i\leq n\right\}\\
&=&\min\{v(c_{i,j,k})+w_i(b_{i,k})\mid1\leq i\leq n,\,1\leq j\leq f_i,\, 1\leq k\leq e_i\}
\end{eqnarray*}
This shows the equality in the theorem holds.

We now show the linear independence of the elements.
If some $c_{i,j,k}$ is nonzero, the rigt hand side will not evaluate to $\infty$.
This then implies $$w_l\left(\sum_{i=1}^n\sum_{j=1}^{f_i}\sum_{k=1}^{e_i}c_{i,j,k}a_{i,j}b_{i,k}\right)\neq\infty$$
for some $l$.
Therefore $$\sum_{i=1}^n\sum_{j=1}^{f_i}\sum_{k=1}^{e_i}c_{i,j,k}a_{i,j}b_{i,k}\neq0.$$
\end{proof}
\begin{cor}{\label{fun}}(Fundamental inequality)
Let $v$ be a valuation on a field $K$ and let $L$ be a finite extension of $K$.
Let $w_1,\ldots,w_n$ be distinct extensions of $v$ to $L$.
$$\sum_{i=1}^ne(w_i|v)f(w_i|v)\leq [L:K]$$
\end{cor}
\begin{proof}
We use \Cref{weak} and \Cref{approx} to find elements satisfying the conditions of \Cref{fe}.
This gives us $\sum_{i=1}^ne(w_i|v)f(w_i|v)$ $K$-linearly independent elements.
The inequality now follows.
\end{proof}

\clearpage
\printbibliography

\end{document}